\documentclass[12pt]{article}
\usepackage{preprint}

\usepackage{amsmath,amssymb}
\usepackage{graphicx}
\usepackage[colorlinks=true,citecolor=black,linkcolor=black,urlcolor=blue]{hyperref}

\dateline{May 29, 2019}


\Copyright{The author. Released under the CC BY license (International 4.0).}

\title{Partition function of the cyclic group}

\author{Steven S Poon\\
\small\tt steven.s.poon@gmail.com\\}

\begin{document}

\maketitle


\begin{abstract}
	This paper addresses the problem of finding $Q_{m,t}\left(n\right)$, the number of possible ways to partition any member $n$ of the cyclic group $\mathbb{Z}/m\mathbb{Z}$ into $t$ distinct parts. When $m$ is odd, it was previously known that the number of partitions of the identity element $0\bmod m$ with distinct parts is equal to the number of possible bi-color necklaces with $m$ beads. This paper will expand upon this result by showing the equivalence between $Q_{m,t}\left(n\right)$ and the number of bi-color necklaces meeting certain periodicity requirements, even when $m$ is even.
\end{abstract}

\section{Introduction}\label{Sec:IntroTop}

The problem of finding the number of possible ways to partition a positive integer with various limitations and conditions on the parts is a well-studied one since Euler wrote down some of the seminal results. There is no reason why one cannot generalize this problem to any abelian group. Given an abelian group $G$ and a fixed member $n$, the group partition function can be defined as the number of distinct multisets of members of $G$ such that the group sum is equal to $n$. Although such a broad scope is just described, this paper will focus on the specific case where $G = \mathbb{Z}/m\mathbb{Z}$, isomorphic to the cyclic group of $m$ elements. The partitions will be required to have distinct parts.

The task of counting the number of subsets of $\mathbb{Z}/m\mathbb{Z}$ with a sum equal to $0\bmod m$, where $m$ is odd, was given as an exercise in \cite{Stanley}, which noted an interesting fact that the number of these subsets matches the number of bi-color necklaces of $m$ beads. The more recent work \cite{Chan} expanded upon this result and showed that a similar correspondence holds when the beads of the necklaces are allowed to have more than two colors.

In both of the works above, the case for even values of $m$ remains unexplained. Also, the partitioning of members of $\mathbb{Z}/m\mathbb{Z}$ other than the identity element was not worked out explicitly. This paper will seek to cover these two aspects for the case when repeating parts are not allowed. In particular, the function $Q_{m,t}\left(n\right)$, which is defined here as the partition function of $n$ as a member of $\mathbb{Z}/m\mathbb{Z}$ with $t$ distinct parts, will be analyzed in depth. By doing so, one can obtain interesting results regarding two different formulations of the same problem:

\begin{enumerate}
	\item Urn model. An urn contains $m$ marble balls, each labeled with a unique integer from the set $\{0,\ldots,m-1\}$. A fixed number $t$ of these marble balls are chosen from the urn at random and without replacement. Let numbers on the drawn balls form the set $\{a_1,\ldots,a_t\}$, and the statistic $N$ is calculated as $N=a_1+\ldots+a_n$. What is the probability of obtaining a certain value of $N\bmod m$ after a random drawing?
	\item Bi-color necklaces. Construct a bijection between the possible ways to partition $n\bmod m$ into $t$ distinct parts and the possible bi-color necklaces of $t$ black beads and $m-t$ white beads meeting certain periodicity requirements.
\end{enumerate}

The key findings regarding $Q_{m,t}\left(n\right)$ applicable to these scenarios are summarized below.

\subsection{Statements of Main Results}\label{SubSec:IntroMainResults}

\begin{theorem}
	\label{Thm:CmPartixForm1}
	Let $m$ and $t$ be integers such that $m>0$, and $n$ be any integer modulo $m$, then:
	
	$$Q_{m,t}\left(n\right)=\frac{1}{m}\sum_{d\mid\left(m,t\right)} \left(-1\right)^{t\left(d+1\right)/d}\binom{m/d}{t/d}c_d\left(n\right)$$
\end{theorem}

The quantity $c_d\left(n\right)$ is Ramanujan's sum. Recall that, if $d$ is a divisor of $m$, then $c_d\left(n\right)=c_d\left(n'\right)$ if $n\equiv n'\bmod m$. Hence the domain of $c_d\left(n\right)$ will be treated as $\mathbb{Z}/m\mathbb{Z}$ whenever appropriate.

If one plays a game based on the urn model described above, it becomes interesting to understand exactly how $Q_{m,t}\left(n\right)$ behaves as a function of $n$. Particularly one might want to know which number $n\bmod m$ is most likely to win. The answer is given in Theorem \ref{Thm:PartixMax}.

\begin{theorem}
	\label{Thm:PartixMax}
	For fixed integers $m$ and $t$ such that $m>0$ and $0\leq t\leq m$, the function $Q_{m,t}\left(n\right)$ is maximized precisely when the values of $n$ meet one of the following requirements:
	\begin{enumerate}
		\setlength\itemsep{0em}
		\item $\left(n,t,m\right)=\left(t,m\right)$ if $\left(t,m\right)$ is odd, or
		\item $\left(n,t,m\right)=\left(t,m\right)$ if $\left(t,m\right)$ is even and $v_2\left(\left(t,m\right)\right)<v_2\left(t\right)$, or
		\item $\left(n,t,m\right)=\left(t,m\right)/2$ if $\left(t,m\right)$ is even and $v_2\left(\left(t,m\right)\right)=v_2\left(t\right)$.
	\end{enumerate}
\end{theorem}

The function $v_2\left(t\right)$ is the 2-adic valuation of $t$.

The correspondence between partitions of $n$ modulo $m$ with distinct parts and necklaces can be expanded to cover $n\not\equiv 0\bmod m$ and even values of $m$ using Theorem \ref{Thm:NecklacePartixBijection} below. As a reminder, a necklace is periodic if one can break it at $u+1$ points to create $u$ contiguous segments, such that the pattern of black and white beads as an ordered list is the same for all segments. The largest value of $u$ possible will be referred to as the frequency of the necklace within this paper.

\begin{theorem}
	\label{Thm:NecklacePartixBijection}
	Let $m$ and $t$ be integers such that $m>0$ and $0\leq t\leq m$, then:
	\begin{enumerate}
		\setlength\itemsep{0em}
		\item $Q_{m,t}\left(n\right)=\left|\mathcal{N}_{m,t,\left< n\right>}\right|$ whenever the conditions $v_2\left(m\right)\geq v_2\left(t\right)\geq 1$ are not met, or
		\item $Q_{m,t}\left(n\right)=\left|\mathcal{N}_{m,t,\left< n\right>}\right|$ if $v_2\left(n\right)<v_2\left(t\right)-1$, or
		\item $Q_{m,t}\left(n\right)=\left|\mathcal{N}_{m,t,\left< 2n\right>}\right|$ if $v_2\left(n\right)=v_2\left(t\right)-1$, or
		\item $Q_{m,t}\left(n\right)=\left|\mathcal{N}_{m,t,\left< (t,n)/2\right>}\right|$ if $v_2\left(n\right)\geq v_2\left(t\right)$.
	\end{enumerate}
\end{theorem}

The quantity $\left|\mathcal{N}_{m,t,\left< n\right>}\right|$ is the number of possible necklaces with $t$ black beads and $m-t$ white beads, such that the frequency of the necklace is a divisor of $n$.

\section{Generating Function Approach}\label{Sec:MainDerive}

All results shown in this paper are based on Theorem \ref{Thm:CmPartixForm1}. This section is dedicated to proving the theorem using a generating function approach.

\subsection{Some Definitions}\label{SubSec:MainDef}

It is useful to start with some definitions to allow the discussion below to proceed smoothly.

\begin{definition}
	\label{Def:IntPartixFunc}
	Let $m$, $t$, and $n$ be integers. The quantity $q_{m,t}\left(n\right)$ is defined recursively as:
	$$q_{m,t}\left(n\right)=q_{m-1,t-1}\left(n-t\right)+q_{m-1,t}\left(n-t\right)+\left[n=m=t=0\right]$$
	with the starting condition $q_{m,t}\left(n\right)=0$ whenever $m$, $t$, or $n$ is less than 0.
\end{definition}

As is well-known, this is simply the partition function of the positive integer $n$ with $t$ distinct parts such that no part is larger than $m$ when all the parameters are positive.

\begin{definition}
	\label{Def:ModPartixFunc}
	Let $m$ and $t$ be integers, $n$ be an integer modulo $m$, and $n'$ be the integer $0\leq n'<m$ from the equivalent class of integers represented by $n$. The quantity $Q_{m,t}^*(n)$ is defined as:
	$$Q_{m,t}^*\left(n\right)=\sum_{j}q_{m-1,t}\left(mj+n'\right)$$
\end{definition}

It is clear that, when $m$ and $t$ are positive, $Q_{m,t}^*\left(n\right)$ is simly the partition function of $n\bmod m$ with $t$ distinct parts such that no part is equal to the identity element $0\bmod m$. The last restriction is the only difference between $Q_{m,t}^*\left(n\right)$ and $Q_{m,t}\left(n\right)$.

\begin{proposition}
	\label{Prop:ModPartixConv}
	Let $m$ and $t$ be integers, and $n$ be an integer modulo $m$. Then $Q_{m,t}\left(n\right)=Q_{m,t-1}^*\left(n\right)+Q_{m,t}^*\left(n\right)$.
\end{proposition}

\begin{proof}
	This is because the partitions counted by $Q_{m,t}\left(n\right)$ can be separated into two ``clouds": those with $0\bmod m$ as a part and those without. The first and second clouds have obvious one-to-one relationships with the partitions counted by $Q_{m,t-1}^*\left(n\right)$ and $Q_{m,t}^*\left(n\right)$, respectively.
\end{proof}

This fact makes $Q_{m,t}^*\left(n\right)$ a highly useful intermediate quantity for the purpose of this work.

\begin{definition}
	The following two short-hands will be used throughout this section:
	
	$$\ddot{Q}_{m,t,s}\left(n\right)=\sum_{u\equiv t\bmod s}Q_{m,u}\left(n\right)$$
	
	$$\ddot{Q}_{m,t,s}^*\left(n\right)=\sum_{u\equiv t\bmod s}Q_{m,u}^*\left(n\right)$$
\end{definition}

The reason for this will be apparent shortly.

\subsection{The Polynomial $F_m\left(x,y,z\right)$}\label{SubSec:FPolynomial}

\begin{definition}
	\label{Def:FPolynomial}
	Let $m$ be a positive integer greater than 1, and $x$, $y$, and $z$ be complex numbers. Then the function $F_m\left(x,y,z\right)$ is given by the product:
	$$F_m(x,y,z)=\prod_{j=1}^{m-1}\left(x+zy^j\right)$$
\end{definition}

It will be useful to expand the allowed values of $m$ to all non-negative integers with the conventions $F_0\left(x,y,z\right)=0$ and $F_1\left(x,y,z\right)=1$.

The function $F_m\left(x,y,z\right)$ is often used when studying integer partitions. When $x=1$ and $m$ indefinitely large, it is the same as the generation function Euler used in his well-known correspondence with Philip Naud\'e. Thus the identity below requires no further explanation.

\begin{proposition}
	\label{Prop:FGenIntPartix}
	The polynomial $F_m(1,y,z)$ is the generating function for $q_{m-1,t}\left(n\right)$, in the sense that:
	$$F_m\left(1,y,z\right)=\sum_{n,t}q_{m-1,t}\left(n\right)y^nz^t$$
\end{proposition}

A more relevant fact for the purpose of this paper comes about from a slight modification of the statement above.

\begin{proposition}
	\label{Prop:FJDFTPair}
	Let $m$, $t$, $s$, and $u$ be integers, $\lambda_m=e^{2\pi i/m}$, and the quantity $J_{m,t,s}\left(u\right)$ be defined as:
	$$J_{m,t,s}\left(u\right)=\sum_{n\in\mathbb{Z}/m\mathbb{Z}}\ddot{Q}_{m,t,s}^*\left(n\right)\lambda_m^{un}$$
	
	Then $J_{m,t,s}\left(u\right)$ as a function of $t$ is the discrete Fourier transform of $F_m\left(1,\lambda_m^u,\lambda_s^v\right)/s$ as a function of $v$. Specifically:
	
	$$F_m\left(1,\lambda_m^u,\lambda_s^v\right)=\sum_{t=0}^{s-1}J_{m,t,s}\left(u\right)\lambda_s^{vt}$$
\end{proposition}

This follows quite naturally from the definitions of $F_m\left(1,y,z\right)$ and $\ddot{Q}_{m,t,s}^*\left(n\right)$.

Note that $\ddot{Q}_{m,t,s}^*\left(n\right)$ itself as a function of $n$ is the discrete Fourier transform of $J_{m,t,s}\left(u\right)/m$ as a function of $u$. Thus one can obtain an expression for $\ddot{Q}_{m,t,s}^*\left(n\right)$ from an expression for $F_m\left(1,\lambda_m^u,\lambda_s^v\right)$. In the next sub-section, it will be shown that a useful form for $F_m\left(1,\lambda_m^u,\lambda_s^v\right)$ can be derived. A clear line of attack towards proving Theorem \ref{Thm:CmPartixForm1} will then be formed.

\subsection{Evaluating $F_m\left(1,\lambda_m^u,z\right)$}\label{SubSec:FPolynomialEval}

The following lemma will be very useful shortly.

\begin{proposition}
	\label{Prop:DivisorEndo}
	Let $m$ and $u$ be positive integers, and $\mathcal{D}_m$ be the set of divisors of $m$. Also define the function $H_{m,u}:\mathcal{D}_m\rightarrow\mathcal{D}_m$ as $H_{m,u}\left(d\right)=d/\left(u,d\right)$. Then the following statements are true:
	\begin{enumerate}
		\setlength\itemsep{0em}
		\item The image of $H_{m,u}$ is the set of divisors of $m/\left(u,m\right)$.
		\item The preimage of $d'$, denoted as $H_{m,u}^{-1}\left(d'\right)$, is the set $\{d\in\mathcal{D}_m:d=d_\perp u_\parallel d',\forall d_\perp\mid\left(m,u_\perp\right)\}$. The quantity $u_\perp$ is the product of all prime factors of $u$ coprime with $d'$, and $u_\parallel=u/u_\perp$.
	\end{enumerate}
\end{proposition}

\begin{proof}
	Consider the prime factorization of $d=p_1^{D_1}\ldots p_\omega^{D_\omega}$, where $\{p_1,\ldots,p_\omega\}$ is the set of distinct prime factors and $D_1,\ldots,D_\omega$ are non-zero. Then $u$ can be written as $u=u_*p_1^{U_1}\ldots p_\omega^{U_\omega}$, where $u_*$ is the product of all prime factors of $u$ coprime with $d$, and $U_1,\ldots,U_\omega$ are either zero or positive. Thus one can write:
	$$d'=d/\left(u,d\right)=p_1^{D_1-\min\left(D_1,U_1\right)}\ldots p_\omega^{D_\omega-\min\left(D_\omega,U_\omega\right)}$$

	One can conclude from this that if $d'$ is in the image of $H_{m,u}$, then any divisor of $d'$ is also in the image of $H_{m,u}$, as one can easily reduce the exponent associated with $p_j$ for any $j$ by dividing $d$ by $p_j$ when appropriate. Since $H_{m,u}\left(m\right)=m/\left(u,m\right)$, this value and its divisors are clearly all within the image of $H_{m,u}$. The fact that $m$ already contains the most number of prime factors possible suggests that the first statement of this proposition must be true.
	
	To show the second statement, notice that if $d=d_\perp u_\parallel d'$ as stated, then $\left(u,d\right)=\left(u_\perp,d_\perp\right)\left(u_\parallel,u_\parallel d'\right)=d_\perp u_\parallel$, and so $d/\left(u,d\right)=d'$ for all valid values of $d$. The claim clearly appears to be plausible, but one still needs to show that no other value of $d$ maps to $d'$ through the function $H_{m,u}$.
	
	To this end, let the prime factorization of $d$ be rewritten as $d=d_\perp p_1^{D_1}\ldots p_\omega^{D_\omega}$, where $D_1,\ldots,D_\omega$ are non-zero and $d_\perp$ is the product of all prime factors coprime to $d'$. Also let $d'=p_1^{D'_1}\ldots p_\omega^{D'_\omega}$, where $D'_1,\ldots,D'_\omega$ are required to be non-zero. Finally, let $u=u_\perp p_1^{U_1}\ldots p_\omega^{U_\omega}$, where $U_1,\ldots,U_\omega$ can be zero or positive.
	
	The fact that $d'=d/\left(u,d\right)$ means $D'_j=D_j-\min\left(D_j,U_j\right)>0$ for all $1\leq j\leq\omega$. But if $U_j\geq D_j$, then $D'_j=0$ which conflicts with the requirement that $D'_j>0$. Thus $D_j>U_j$, and consequently $D'_j+U_j=D_j$ for all $1\leq j\leq\omega$. This means $u_\parallel d'$ must divide $d$. In fact, $d/\left(u_\parallel d'\right)=d_\perp$ can only contain prime factors coprime to $d'$.
	
	At this point, it remains to show that $d_\perp$ can only take on values that are divisors of $\left(m,u_\perp\right)$. If this is not true, then $\left(u,d\right)=\left(u_\perp,d_\perp\right)u_\parallel$, where $\left(u_\perp,d_\perp\right)\neq d_\perp$, and so $d/\left(u,d\right)=\left(d_\perp/\left(u_\perp,d_\perp\right)\right)d'$, which is not equal to $d'$.
\end{proof}

One now has all the tools needed to consider the main result of this sub-section.

\begin{proposition}
	\label{Prop:FNewForm}
	Let $m$ and $u$ be positive integers, $\alpha=m/\left(m,u\right)$, and $\beta=\left(m,u\right)$, then:
	$$F_m\left(1,\lambda_m^u,z\right)=\frac{(1-\left(-z\right)^\alpha)^\beta}{1+z}$$
\end{proposition}

\begin{proof}
	Let $\left(\mathbb{Z}/m\mathbb{Z}\right)^\times$ be the set of positive integers less than and coprime with $m$, then Definition \ref{Def:FPolynomial} can be rewritten as:
	$$F_m\left(1,\lambda_m^u,z\right)=\left(-z\right)^{m-1}\prod_{d\mid m,d\neq 1}\prod_{j\in\left(\mathbb{Z}/d\mathbb{Z}\right)^\times}(-z^{-1}-\lambda_{D}^{Uj)})$$
	where $U=u/\left(u,d\right)$ and $D=d/\left(u,d\right)$.
	
	Let $\{j_1,\ldots,j_{\varphi\left(d\right)}\}=\left(\mathbb{Z}/d\mathbb{Z}\right)^\times$ and $\{j'_1,\dots,j'_{\varphi\left(D\right)}\}=\left(\mathbb{Z}/D\mathbb{Z}\right)^\times$, where $\varphi$ is the Euler totient function. Using elementary facts regarding the group $\left(\mathbb{Z}/m\mathbb{Z}\right)^\times$, one knows that the ordered list $\left(Uj_1,\ldots,Uj_{\varphi\left(d\right)}\right)$, when each entry is interpreted as a member of $\left(\mathbb{Z}/D\mathbb{Z}\right)^\times$, contains $\varphi\left(d\right)/\varphi\left(D\right)$ copies each of $j'_1$, ..., and $j'_{\varphi\left(D\right)}$. It follows that:
	$$F_m\left(1,\lambda_m^u,z\right)=\frac{\left(-z\right)^m}{1+z}\prod_{d\mid m}\left(\Phi_{d/\left(u,d\right)}\left(-z^{-1}\right)\right)^{\varphi\left(d\right)/\varphi\left(d/\left(u,d\right)\right)}$$
	where $\Phi_d\left(z\right)$ is the $d$~\textsuperscript{th} cyclotomic polynomial.
	
	It is possible for multiple factors of this product to have the same value of $d/\left(u,d\right)$. As an example, every divisor of $\left(m,u\right)$ is also a divisor of $m$. Thus $\left(u,d\right)=d$ for all values of $d$ that are divisors of $\left(m,u\right)$, such that $d/\left(u,d\right)$ is always equal to 1 in this case.
	
	This is where Proposition \ref{Prop:DivisorEndo} becomes useful. It is now clear that $d'=d/\left(u,d\right)$ can only take on values that are divisors of $m/\left(u,m\right)$, and the values of $d$ that map to a single fixed value of $d'$ are in the form $d=d_\perp u_\parallel d'$, where $d_\perp$ is any divisor of $\left(m,u_\perp\right)$. Thus the equation above can be rewritten as:
	$$F_m\left(1,\lambda_m^u,z\right)=\frac{\left(-z\right)^m}{1+z}\prod_{d'\mid\left(m/\left(m,u\right)\right)}\left(\Phi_{d'}\left(-z^{-1}\right)\right)^{X_{d'}}$$
	$$X_{d'}=\frac{1}{\varphi\left(d'\right)}\sum_{d_\perp\mid\left(m,u_\perp\right)}\varphi\left(d_\perp u_\parallel d'\right)$$
	
	With a small amount of effort, one can show with standard properties of the Euler totient function that $X_{d'}$ simply evaluates to $X_{d'}=\left(m,u\right)$. Thus:
	$$F_m\left(1,\lambda_m^u,z\right)=\frac{(-z)^m}{1+z}\prod_{d'\mid\left(m/\left(m,u\right)\right)}\left(\Phi_{d'}\left(-z^{-1}\right)\right)^{\left(m,u\right)}$$
	Applying the well-known identity $\prod_{d\mid m}\Phi_d\left(z\right)=z^m-1$ immediately results in the desired form.
\end{proof}

A corollary of this result is the following.

\begin{proposition}
	\label{Prop:FNewForm2}
	Let $m$ and $u$ be positive integers, $\alpha=m/\left(m,u\right)$, and $\beta=\left(m,u\right)$, then:
	$$F_m\left(1,\lambda_m^u,z\right)=\sum_{j}\binom{\beta-1}{\left\lfloor{j/\alpha}\right\rfloor}\left(-1\right)^{\left\lfloor{j/\alpha}\right\rfloor+j}z^j$$
\end{proposition}

\begin{proof}
	The quantity $F_m\left(1,\lambda_m^u,z\right)=\left(1-\left(-z\right)^\alpha\right)^\beta/\left(1+z\right)$, derived just above, is a polynomial as a function of $z$. This is because it can be written as the product of the factors $\left(1-\left(-z\right)^\alpha\right)/\left(1+z\right)$ and $\left(1-\left(-z\right)^\alpha\right)^{\beta-1}$. The first factor is reminiscent of the geometric sum and can be written as $1+\left(-z\right)+\left(-z\right)^2+\ldots+\left(-z\right)^{\alpha-1}$. The second factor can be expanded using the binomial theorem, resulting in the form $\binom{\beta-1}{0}1-\binom{\beta-1}{1}\left(-z\right)^\alpha+\binom{\beta-1}{2}\left(-z\right)^{2\alpha}+\ldots+\left(-1\right)^{\beta-1}\binom{\beta-1}{\beta-1}\left(-z\right)^{\alpha\left(\beta-1\right)}$. It is easy to see that multiplying these two factors together gives the desired form.
\end{proof}

\subsection{Formula for $J_{m,t,s}\left(u\right)$}\label{SubSec:JFunction}

The result from Proposition \ref{Prop:FNewForm2} allows one to generate new identities for $J_{m,t,s}\left(u\right)$, as it is related to $F_m\left(1,\lambda_m^u,\lambda_s^t\right)$ through the discrete Fourier transform, which can be inverted.

\begin{proposition}
	\label{Prop:JNewForm}
	Let $s$, $m$, and $u$ be positive integers, $t$ be an integer, $\alpha=m/\left(m,u\right)$, and $\beta=\left(m,u\right)$, then the function $J_{m,t,s}\left(u\right)$ can be written as $J_{m,t,s}\left(u\right)=\Upsilon_{s,t}\left(\alpha,\beta\right)$, where:
	$$\Upsilon_{s,t}\left(\alpha,\beta\right)=\sum_{j\equiv t\bmod s}\left(-1\right)^{\left\lfloor{j/\alpha}\right\rfloor+j}\binom{\beta-1}{\left\lfloor{j/\alpha}\right\rfloor}$$
\end{proposition}

\begin{proof}
	It was already observed that $\left(1/s\right)F_m\left(1,\lambda_m^u,\lambda_s^t\right)$ is the inverse discrete Fourier transform of $J_{m,t,s}\left(u\right)$. Thus one can use the result from Proposition \ref{Prop:FNewForm2} to write:
	$$\frac{1}{s}\sum_{v=0}^{s-1}F_m\left(1,\lambda_m^u,\lambda_s^v\right)\lambda_s^{-vt}=\frac{1}{s}\sum_{j=0}^{m-1}\binom{\beta-1}{\left\lfloor{j/\alpha}\right\rfloor}\left(-1\right)^{\left\lfloor{j/\alpha}\right\rfloor+j}\sum_{v=0}^{s-1}\lambda_s^{v\left(j-t\right)}$$
	But the inner summation is equal to zero unless $s$ divides $j-t$, in which case it is equal to $s$. Thus the right hand side above simplifies to the desired form.
\end{proof}

\subsection{Formula for $\ddot{Q}_{m,t,s}^*\left(n\right)$}\label{SubSec:QStarFunction}

With the result from Proposition \ref{Prop:JNewForm} available, one can now write a new expression for $\ddot{Q}_{m,t,s}^*\left(n\right)$.

\begin{proposition}
	\label{Prop:ModPartixSkipStarNewForm}
	The function $\ddot{Q}_{m,t,s}^*\left(n\right)$ can be written as:
	$$\ddot{Q}_{m,t,s}^*\left(n\right)=\frac{1}{m}\sum_{d\mid m}\Upsilon_{s,t}\left(d,\frac{m}{d}\right)c_d\left(n\right)$$
\end{proposition}

\begin{proof}
	By definition, $\left(1/m\right)J_{m,t,s}\left(u\right)$ as a function of $u$ is the inverse discrete Fourier transform of $\ddot{Q}_{m,t,s}^*\left(n\right)$ as a function of $n$. This fact allows one to write $\ddot{Q}_{m,t,s}^*\left(n\right)=\left(1/m\right)\sum_{u=0}^{m-1}J_{m,t,s}\left(u\right)\lambda_m^{-un}$. Using the result from Proposition \ref{Prop:JNewForm}, one can rewrite this as $\ddot{Q}_{m,t,s}^*\left(n\right)=\left(1/m\right)\sum_{u=0}^{m-1}\Upsilon_{s,t}\left(m/\left(m,u\right),\left(m,u\right)\right)\lambda_m^{-un}$.
	
	Notice that $\left(1/m\right)J_{m,t,s}\left(u\right)$ is a function of $\left(m,u\right)$. The value of $\left(m,u\right)$ is necessarily a divisor of $m$. Thus one may reorganize the sum such that the index is over all divisors of $m$, such that all appearances of $\left(m,u\right)$ are replaced with $d$, and $u$ with $jd$, where $j$ is coprime with $m$. This gives the updated form:
	$$\ddot{Q}_{m,t,s}^*\left(n\right)=\left(1/m\right)\sum_{d\mid m}\Upsilon_{s,t}\left(m/d,m\right)\sum_{j\in\left(\mathbb{Z}/d\mathbb{Z}\right)^\times}\lambda_{m/d}^{-jn}$$
	The inner summation is well-known to be equivalent to Ramanujan's sum. The desired form is then obtained by replacing the index $d$ by $m/d$, which plainly has the effect of reshuffling the terms within the summation and does not change the result.
\end{proof}

Furthermore, one can obtain a similar expression for $\ddot{Q}_{m,t,s}\left(n\right)$.

\begin{proposition}
	\label{Prop:ModPartixSkipNewForm}
	Let $s\geq 1$, then the function $\ddot{Q}_{m,t,s}\left(n\right)$ can be written as:
	$$\ddot{Q}_{m,t,s}\left(n\right)=\frac{1}{m}\sum_{d\mid m}c_d\left(n\right)\sum_{j\equiv t\bmod s}\left[d\mid j\right]\left(-1\right)^{j\left(d+1\right)/d}\binom{m/d}{j/d}$$
\end{proposition}

\begin{proof}
	As mentioned in Proposition \ref{Prop:ModPartixConv}, $Q_{m,t}\left(n\right)=Q_{m,t-1}^*\left(n\right)+Q_{m,t}^*\left(n\right)$, and so naturally it is also true that $\ddot{Q}_{m,t,s}\left(n\right)=\ddot{Q}_{m,t-1,s}^*\left(n\right)+\ddot{Q}_{m,t,s}^*\left(n\right)$. Applying the result from Proposition \ref{Prop:ModPartixSkipStarNewForm} and the definition for $\Upsilon_{s,t}\left(\alpha,\beta\right)$ gives the following expression:
	$$\ddot{Q}_{m,t,s}\left(n\right)=\frac{1}{m}\sum_{d\mid m}c_d\left(n\right)\sum_{j\equiv t\bmod s}\left(\left(-1\right)^{\left\lfloor{\frac{j}{d}}\right\rfloor+j}\binom{\frac{m}{d}-1}{\left\lfloor{\frac{j}{d}}\right\rfloor}-\left(-1\right)^{\left\lfloor{\frac{j-1}{d}}\right\rfloor+j}\binom{\frac{m}{d}-1}{\left\lfloor{\frac{j-1}{d}}\right\rfloor}\right)$$

	For the case $s>1$, the summand of the inner summation is clearly equal to zero when $d$ does not divide $j$. When $d$ divides $j$, the summand is easily shown to be equal to $\left(-1\right)^{j\left(d+1\right)/d}\binom{m/d}{j/d}$.
	
	The case $s=1$ can be shown through direct evaluation.
\end{proof}

\subsection{Proof for Theorem \ref{Thm:CmPartixForm1}}\label{SubSec:QFunction}

The proof for Theorem \ref{Thm:CmPartixForm1} follows quite naturally from the results above.

Note that $\ddot{Q}_{m,t,s}\left(n\right)=Q_{m,t}\left(n\right)$ when $s>m$ and $0\leq t<m$. Taking the result from Proposition \ref{Prop:ModPartixSkipNewForm} and setting $s=m+1$ immediately results in the desired expression.

Another way to establish the validity of Theorem \ref{Thm:CmPartixForm1} is to first notice the fact that, assuming that the proposed equation for $Q_{m,t}\left(n\right)$ is true, then the sum $\sum_{u\equiv 0\bmod s}Q_{m,u}\left(n\right)$ does indeed result in the correct form:
$$\ddot{Q}_{m,0,s}\left(n\right)=\frac{1}{m}\sum_{d\mid m}c_d\left(n\right)\sum_{u}\left(-1\right)^{su\left(d+1\right)/\left(d,s\right)}\binom{m/d}{su/\left(d,s\right)}$$
which is consistent with the result from Proposition \ref{Prop:ModPartixSkipNewForm}.

However, one still needs to show that this is not just a coincidence. It is clear that $\ddot{Q}_{m,0,s}\left(n\right)=\sum_{u\equiv 0\bmod s}Q_{m,u}\left(n\right)$ forms a system of equations for $s\in\{1,\ldots,m\}$. The ordered list of values $\mathbf{q}=\left(Q_{m,0}\left(n\right),Q_{m,1}\left(n\right),\ldots,Q_{m,m-1}\left(n\right)\right)$ can be related to $\mathbf{q'}=(\ddot{Q}_{m,0,1}\left(n\right),\ddot{Q}_{m,0,2}\left(n\right),\ldots,\ddot{Q}_{m,0,m}\left(n\right))$ through the matrix equation $\mathbf{M_mq}=\mathbf{q'}$, where $\mathbf{q}$ and $\mathbf{q'}$ are treated as column vectors. The matrix $\mathbf{M_m}$ is an $m$-by-$m$ matrix such that the entry $\left(\mathbf{M_m}\right)_{s,v}$ is equal to 1 when $s$ divides $v$, and equal to 0 otherwise. It can be shown quite easily through induction that $\det\left(\mathbf{M_m}\right)=\left(-1\right)^{m-1}$, such that $\mathbf{M_m}$ is always invertible. Thus the fact that the sum $\sum_{u\equiv 0\bmod s}Q_{m,u}\left(u\right)$ gives the correct form for $\ddot{Q}_{m,0,s}\left(n\right)$ is sufficient proof that the proposed form for $Q_{m,t}\left(n\right)$ is correct.

\section{Urn Model}\label{Sec:UrnModel}

\subsection{Probability Distribution}\label{SubSec:Probability}

Imagine a game in which $m$ objects labeled with the integers 0, 1, \ldots, and $m-1$ are prepared ahead of time. A fixed number $t$ of these objects are then randomly chosen without replacement. Such an arrangement is used by various lottery games. One can also imagine replacing the dice used in popular board games with this drawing mechanism.

For each drawing, it is always possible to calculate a statistic defined as $N'=\textrm{rem}\left(N,m\right)$, where $N=a_1+\ldots+a_t$, $\{a_1,\ldots,a_t\}$ is the set of numbers drawn, and $\textrm{rem}\left(N,m\right)$ is the remainder function with $m$ as the divisor. One may want to do that, for example, if the game board contains $m$ spaces labeled consecutively from 0 to $m-1$ arranged in a circular fashion such that the slot labeled $m-1$ is adjacent to the slot labeled 0. One may also want to do that simply for the sake of studying the statistic.

\begin{proposition}
	\label{Prop:UrnProbability}
	Given the urn model discussed above, the probability distribution of the statistic $N'=\textrm{rem}\left(a_1+\ldots+a_t,m\right)$ is equal to $Q_{m,t}\left(n\right)/\binom{m}{t}$, where $n\bmod m$ is any possible value of the statistic.
\end{proposition}

\begin{proof}
	It is plain that the number of ways to obtain a certain fixed number $n$ after calculating the sum $a_1+\ldots+a_t$ is simply the sum of the two integer partition functions $q_{m-1,t}\left(n\right)+q_{m-1,t-1}\left(n\right)$. This is because one of the members of $\{a_1,\ldots,a_t\}$ can be equal to zero. The number of ways to obtain the number $n'$ after calculating $\textrm{rem}\left(n,m\right)$ is simply $\sum_{j}q_{m-1,t}\left(mj+n'\right)+\sum_{j}q_{m-1,t-1}\left(mj+n'\right)$, which, as discussed in Section \ref{SubSec:MainDef}, is equal to $Q_{m,t}^*\left(n\right)+Q_{m,t-1}^*\left(n\right)=Q_{m,t}\left(n\right)$.
	
	To find the probability, it is necessary to find the sum $\sum_{n\in\mathbb{Z}/m\mathbb{Z}}Q_{m,t}\left(n\right)$. This can be obtained from Theorem \ref{Thm:CmPartixForm1} and the well-known fact that $\sum_{n=0}^{d-1}c_d\left(n\right)=\left[d=1\right]$. The result is $\sum_{n\in\mathbb{Z}/m\mathbb{Z}}Q_{m,t}\left(n\right)=\binom{m}{t}$. The desired probability is thus $Q_{m,t}\left(n\right)$ divided by this value.
\end{proof}

\subsection{Some Key Features of the Diagram for $Q_{m,t}\left(n\right)$}\label{SubSec:QDiagram}

Suppose one plays a lottery-like game, in which one wins by guessing the value of $N'=\textrm{rem}\left(a_1+\ldots+a_t,m\right)$ correctly. It is natural to ask if there exists one member or a subset of $\mathbb{Z}/m\mathbb{Z}$ such that the probability of winning is maximized.

Before answering this question, it is useful to examine the properties of $Q_{m,t}\left(n\right)$ further. It was already established in the section above that the set $\{Q_{m,t}\left(0\right),\ldots,Q_{m,t}\left(m-1\right)\}$ is an integer partition of the binomial coefficient $\binom{m}{t}$. It is interesting to note that there is a related way to write the same binomial coefficient as a sum of a set of integers. First, it is necessary to define the following quantity.

\begin{definition}
	\label{Def:AFunction}
	Let $m$ and $t$ be integers such that $m>0$. The quantity $A\left(m,t\right)$ is defined as the sum:
	$$A\left(m,t\right)=\frac{1}{m}\sum_{d\mid\left(m,t\right)}\left(-1\right)^{t\left(d+1\right)/d}\mu\left(d\right)\binom{m/d}{t/d}$$

\end{definition}

The reason for defining this quantity will be seen shortly.

If one attempts to plot the value of $Q_{m,t}\left(n\right)$ versus $n$ on a graph, one gets the image of a ``wall with battlements". This is especially true if one presents the graph as something similar to a Ferrers diagram. This is done by drawing a column of $Q_{m,t}\left(0\right)$ dots or boxes, followed by a column of $Q_{m,t}\left(1\right)$ boxes to the right of the previous column, and so on. The boxes are bottom-justified, as if affected by gravity. The total number of boxes, as discussed previously, is equal to $\binom{m}{t}$

\begin{figure}[ht]
	\centering
	\includegraphics[scale=0.5]{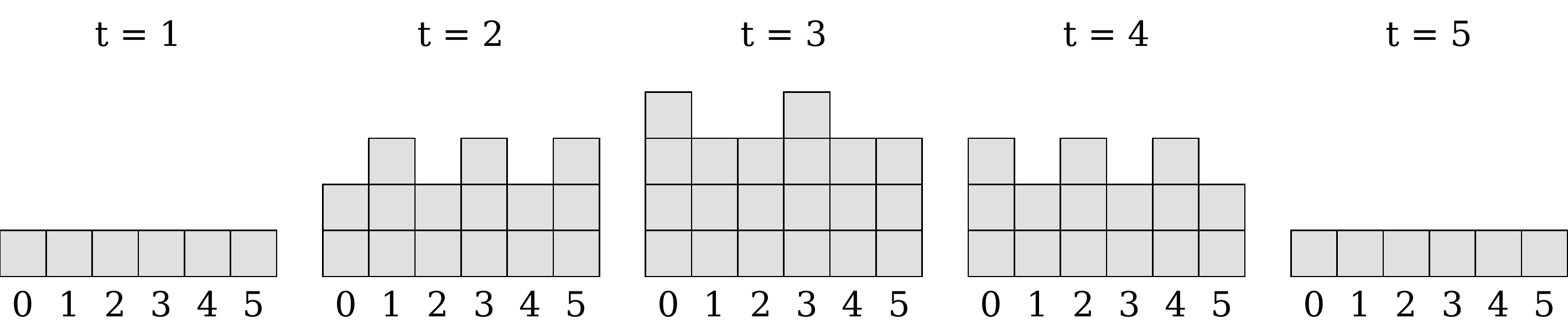}
	\caption{Diagram for $m=6$.\label{fig:Diagram6}}
\end{figure}

The analogy with Ferrers diagrams beckons one to count rows of boxes instead of just columns. Recall that:
$$\left(x-1\right)\left(x-\lambda_m\right)\ldots\left(x-\lambda_m^{m-1}\right)=\sum_{t\geq 0}\sum_{n\in\mathbb{Z}/m\mathbb{Z}}\left(-1\right)^{m-t}Q_{m,m-t}\left(n\right)\lambda_m^nx^t$$
But the left hand side is just $x^m-1$. It follows that $\sum_{n\in\mathbb{Z}/m\mathbb{Z}}Q_{m,t}\left(n\right)\lambda_m^n=0$ for $0<t<m$. Thus if each box on the $n$ \textsuperscript{th} column in the diagram is seen as representing the root of unity $\lambda_m^n$, then the sum of all the boxes of the diagram for $0<t<m$ is equal to zero. Because of the well-known fact that $1+\lambda_m^d+\ldots+\lambda_m^{(m/d-1)d}=0$ for any positive integer $d$ dividing $m$ not equal to $m$ itself, it would seem that there should exist integers $\{A_{m,t,d_1}^*,\ldots,A_{m,t,d_\sigma}^*\}$, where $\{d_1,\ldots,d_\sigma\}$ is the set of proper divisors of $m$, such that the following two equations are met simultaneously for $1\leq t\leq m-1$:

\bigskip
\begin{minipage}{.5\linewidth}
	$$\binom{m}{t}=\sum_{d\mid m,d\neq m}\frac{m}{d}A_{m,t,d}^*$$
\end{minipage}%
\begin{minipage}{.5\linewidth}
	$$Q_{t,m}(n)=\sum_{d\mid m,d\neq m}\left[d\mid n\right]A_{m,t,d}^*$$
\end{minipage}
\smallskip

Before proceeding along this line of thought, it is useful to have the following lemma.

\begin{proposition}
	\label{Prop:MobiusInvVar}
	Let $X\left(m,t\right)$ be a function mapping a pair of positive integers to some subset of the complex plane. Also, let $Y\left(m,t\right)$ be a function meeting the same description. For the following two equations, if one of them is true for all $m$ and $t$, then the other one is also true for all $m$ and $t$.
	$$Y\left(m,t\right)=\sum_{d\mid\left(m,t\right)}\left(-1\right)^{t\left(d+1\right)/d}\frac{m}{d}X\left(\frac{m}{d},\frac{t}{d}\right)$$
	$$X\left(m,t\right)=\frac{1}{m}\sum_{d\mid\left(m,t\right)}\left(-1\right)^{t\left(d+1\right)/d}\mu\left(d\right)Y\left(\frac{m}{d},\frac{t}{d}\right)$$
\end{proposition}

Note that this can be seen as an alternative form of the M\"{o}bius inversion formula. It can be proven similarly, by substituting one equation into another and making use of the identity $\sum_{d\mid D}\mu\left(d\right)=\left[D=1\right]$. Since the proof for the M\"{o}bius inversion formula is well-known, the details will not be worked out here.

A second lemma will also be used.

\begin{proposition}
	\label{Prop:RamanujanSumExp}
	Let $X\left(m,d\right)$ be a function defined for all positive integers $m$ and $d$. Also let $u$ and $v$ be any positive integers, and $Y\left(m,n,u,v\right)$ be defined by the sum:
	$$Y\left(m,n,u,v\right)=\frac{1}{muv}\sum_{d\mid m}c_d\left(n\right)X\left(\frac{mu}{d},dv\right)$$
	Then the following expression is a true statement:
	$$Y\left(m,n,u,v\right)=\sum_{d_1d_2\mid m}\left[d_1\mid n\right]\frac{d_1}{muv}\mu\left(d_2\right)X\left(\frac{mu}{d_1d_2},d_1d_2v\right)$$
\end{proposition}

Note that the summation is over all positive integers $d_1$ and $d_2$ meeting the stated condition.

It is easy to show, using standard properties of Ramanujan's sum, that the claim holds true when $m=1$ or when $m$ is prime. Using the same standard properties, it is possible to show that the desired expression holds true for $m=m'p$ if it is true for all $m$ equal to any divisors of $m'$. Note that the cases for $p$ coprime with $m'$ and $p$ not coprime with $m'$ need to be shown separately. Thus the expression can be shown to be true through induction. Due to the straightforward nature and the space required to develop the steps in full, this will be left as an exercise for the interested reader.

The values of $A_{m,t,d}^*$ can then be easily worked out from the results developed so far.

\begin{proposition}
	\label{Prop:ModPartixFuncNewForm}
	Let $m$ and $t$ be integers, such that $m>0$, then the quantity $A\left(m,t\right)$ agrees with the identities below.
	$$\binom{m}{t}=\sum_{d\mid\left(m,t\right)}\left(-1\right)^{t\left(d+1\right)/d}\frac{m}{d}A\left(\frac{m}{d},\frac{t}{d}\right)$$
	$$Q_{m,t}\left(n\right)=\sum_{d\mid\left(m,t,n\right)}\left(-1\right)^{t\left(d+1\right)/d}A\left(\frac{m}{d},\frac{t}{d}\right)$$
\end{proposition}

Thus it must be true that the hypothetical quantity $A_{m,t,d}^*$ above is given by $A_{m,t,d}^*=\left(-1\right)^{t\left(d+1\right)/d}A\left(m/d,t/d\right)$ and $A_{m,t,m}^*=0$.

\begin{proof}
	The second proposed identity is a consequence of Proposition \ref{Prop:RamanujanSumExp}, which has the following special case:
	$$Q_{m,t}\left(n\right)=\sum_{d_1\mid\left(m,t\right)}\left[d_1\mid n\right]\frac{d_1}{m}\sum_{d_2\mid\left(m/d_1,t\right)}\mu\left(d_2\right)\left(-1\right)^{t\left(d_1d_2+1\right)/\left(d_1d_2\right)}\binom{m/\left(d_1d_2\right)}{t/\left(d_1d_2\right)}$$
	
	It is clear from this equation and Definition \ref{Def:AFunction} that the second proposed identity is true. The first identity then follows directly by inverting Definition \ref{Def:AFunction} using Proposition \ref{Prop:MobiusInvVar}.
\end{proof}

Since $Q_{m,t}\left(n_1\right)=Q_{m,t}\left(n_2\right)$ whenever $\left(n_1,m,t\right)=\left(n_2,m,t\right)$, the various questions regarding the shape of the diagram, such as the left-right symmetry and ``wall with battlements" appearance are explained satisfactorily. The fact that $Q_{m,t}\left(n\right)$ is a constant function of $n$ if and only if $m$ and $t$ are coprime is now also obvious.

It is useful to point out a few basic properties of $A\left(m,t\right)$ at this point.

\begin{proposition}
	Let $m>0$. The following statements are true regarding the quantity $A\left(m,t\right)$:
	\begin{enumerate}
		\setlength\itemsep{0em}
		\item $A\left(m,t\right)$ is non-negative.
		\item $A\left(m,0\right)=\left[m=1\right]$.
		\item $A\left(m,m\right)=\left[m\in\{1,2\}\right]$.
		\item $A\left(m,t\right)=0$ whenever $t>m\geq 1$.
	\end{enumerate}
\end{proposition}

\begin{proof}
	From Proposition \ref{Prop:ModPartixFuncNewForm}, one can set $n=1$ and obtain $Q_{m,t}\left(1\right)=A\left(m,t\right)$. Since $Q_{m,t}\left(1\right)$ is non-negative, $A\left(m,t\right)$ is also non-negative. The second and third statements can be obtained through direct evaluation of Definition \ref{Def:AFunction}. The fourth statement also follows from the property of $Q_{m,1}\left(1\right)$: logically one cannot create a subset of $\mathbb{Z}/m\mathbb{Z}$ with more than $m$ members, and numerically the recursion formula for $q_{m,t}\left(n\right)$ ensures this claim to be true.
\end{proof}

\subsection{Proof of Theorem \ref{Thm:PartixMax} - Finding $n$ That Maximizes $Q_{m,t}\left(n\right)$}\label{SubSec:QDiagramMax}

There is now enough information to answer the question asked at the beginning of this section.

Consider the expression for $Q_{m,t}\left(n\right)$ from the second statement in Proposition \ref{Prop:ModPartixFuncNewForm}. If $\left(t,m\right)$ is odd, then the index $d$ must be odd, and the sign of the summand is always positive. Since $A\left(m,t\right)$ is non-negative, it follows that the maximum value is obtained when the sum has the largest possible number of non-zero terms. This happens when $n$ contains all the divisors of $\left(t,m\right)$. Thus the first statement of Theorem \ref{Thm:PartixMax} is true.

If $\left(t,m\right)$ is even, then the sum for $Q_{m,t}\left(n\right)$ from Proposition \ref{Prop:ModPartixFuncNewForm} can contain negative terms. This happens when the 2-adic valuations of $d$ and $t$ are the same non-zero value. Thus the value of $Q_{m,t}\left(n\right)$ is maximized when the set of divisors of $\left(n,t,m\right)$ is the set of divisors of $\left(t,m\right)$, except for the ones with $v_2\left(t\right)$ as the 2-adic valuation. This explains statements 2 and 3 of Theorem \ref{Thm:PartixMax}.

\section{Relationship with Necklaces}\label{Sec:Necklace}

\subsection{Some Results Regarding Periodic Necklaces}\label{SubSec:NecklaceFreq}

As mentioned in the introduction, a major goal of this paper is to find a way to related partitions of arbitrary elements of $\mathbb{Z}/m\mathbb{Z}$ into distinct parts to certain subsets of bi-color necklaces with $m$ beads, even for the case when $m$ is even. In this section, it will be shown that the results developed thus far allow one to achieve that goal.

Key to this outcome is the concept of the ``frequency" of a necklace, which was already described in the introduction. A necklace with frequency equal to 1 is also said to be aperiodic. As is well-known, there is a bijection between binary Lyndon words and aperiodic bi-color necklaces. An all-black or all-white necklace has $m$ as the frequency, where $m$ is the number of beads in the necklace.

If the frequency $u$ of a necklace is not equal to 1, one can always cut the necklace at $u-1$ equally-spaced points to create $u$ equal segments, and then stitch the two sides of each segment together to form $u$ equivalent smaller necklaces. Alternatively, one can take $u$ equivalent aperiodic necklaces, convert them into equivalent segments by cutting at the same point, and stitch the segments together to create a new larger necklace with $u$ as the new frequency.

\begin{definition}
	\label{Def:NecklaceFreqFunc}
	Let $m$ and $t$ be integers such that $m>0$ and $0\leq t \leq m$. Also let $U=\{u_1,\ldots,u_{\left|U\right|}\}$ be any set of positive integers. Then $\mathcal{N}_{m,t,U}$ is defined as the set of bi-color necklaces of $t$ black beads, $m-t$ white beads, and frequency equal to one of the members of $U$.
\end{definition}

Note that $\mathcal{N}_{m,t,U}$ is the union of the disjoint sets $\mathcal{N}_{m,t,\{u_1\}}$, \ldots, $\mathcal{N}_{m,t,\{u_{\left|U\right|}\}}$, such that $\left|\mathcal{N}_{m,t,U}\right|=\sum_{u\in U}\mathcal{N}_{m,t,\{u\}}$. When $U$ is the set of all divisors of the positive integer $n$, $\mathcal{N}_{m,t,U}$ will be written as $\mathcal{N}_{m,t,\left<n\right>}$.

\begin{proposition}
	\label{Prop:NecklaceFreqFuncNewForm}
	Let $m$, $t$, and $n$ be positive integers such that $m\geq t$, then:
	$$\left|\mathcal{N}_{m,t,\left<n\right>}\right|=\sum_{d\mid\left(m,t,n\right)}\left|\mathcal{N}_{m/d,t/d,\left<1\right>}\right|$$
\end{proposition}

\begin{proof}
	$\mathcal{N}_{m,t,\{u\}}$ is non-empty if and only if $u$ shares common factors with $\left(m,t\right)$, and so $\mathcal{N}_{m,t,\left<n\right>}$ can also be written as $\mathcal{N}_{m,t,\left<\left(m,t,n\right)\right>}$. The number of necklaces within the set $\mathcal{N}_{m,t,\left<n\right>}$ is then clearly the sum of $\left|\mathcal{N}_{m,t,\{d\}}\right|$ over all divisors $d$ of $\left(m,t,n\right)$. But $\left|\mathcal{N}_{m,t,\{d\}}\right|=\left|\mathcal{N}_{m/d,t/d,\{1\}}\right|$ using the necklace cutting argument described previously.
\end{proof}

\begin{proposition}
	\label{Prop:NecklaceFreqFuncNewForm2}
	Let $m$, $t$, and $n$ be positive integers such that $m\geq t$, then:
	$$\left|\mathcal{N}_{m,t,\left<n\right>}\right|=\frac{1}{m}\sum_{d\mid\left(m,t\right)}\binom{m/d}{t/d}c_d\left(n\right)$$
\end{proposition}

\begin{proof}
	Applying Proposition \ref{Prop:RamanujanSumExp} to the right hand side gives:
	$$\frac{1}{m}\sum_{d\mid\left(m,t\right)}\binom{m/d}{t/d}c_d\left(n\right)=\sum_{d_1\mid\left(m,t\right)}\left[d_1\mid n\right]\sum_{d_2\mid\left(\left(m,t\right)/d_1\right)}\frac{d_1}{m}\mu\left(d_2\right)\binom{m/\left(d_1d_2\right)}{t/\left(d_1d_2\right)}$$

	The inner summation is known to be equal to the number of aperiodic bi-color necklaces with $t/d_1$ black beads and $\left(m-t\right)/d_1$ white beads \cite{Mestrovic}, which according to the annotation of this paper is $\left|\mathcal{N}_{m/d_1,t/d_1,\left<1\right>}\right|$. Thus the right hand side is turned into the form $\sum_{d\mid\left(m,t,n\right)}\left|\mathcal{N}_{m/d,t/d,\left<1\right>}\right|$. This is equal to $\left|\mathcal{N}_{m,t,\left<n\right>}\right|$ according to Proposition \ref{Prop:NecklaceFreqFuncNewForm}.
\end{proof}

\subsection{Proof of Theorem \ref{Thm:NecklacePartixBijection}}\label{SubSec:NecklaceBiject}

This theorem can be shown by comparing Proposition \ref{Prop:NecklaceFreqFuncNewForm2} with Theorem \ref{Thm:CmPartixForm1}. One can easily notice that the two equations would have been the same if not for the sign factor in Theorem \ref{Thm:CmPartixForm1}.

Indeed, they are the same when the said sign is always positive. This happens when $m$ is odd, or when $m$ is even but the 2-adic valuation of $t$ is greater than that of $m$. Thus the first statement in Theorem \ref{Thm:NecklacePartixBijection} is true.

If $v_2\left(m\right)\geq v_2\left(t\right)\geq 1$, then $t$ can be written as $t=2^vT$, where $T$ is odd and $v\geq 1$. Also, let $m=2^vM$ and $N=n/2^{v_2\left(n\right)}$. Subtracting Theorem \ref{Thm:CmPartixForm1} from Proposition \ref{Prop:NecklaceFreqFuncNewForm2} gives:

$$Q_{m,t}\left(n\right)=\left|\mathcal{N}_{m,t,\left<n\right>}\right|-\frac{c_{2^v}\left(n\right)}{2^{v-1}}\left|\mathcal{N}_{M,T,\left<N\right>}\right|$$

It is a standard property of Ramanujan's sum that $c_{2^v}\left(n\right)=0$ when $v_2\left(n\right)<v_2\left(t\right)-1$. Thus the second statement in Theorem \ref{Thm:NecklacePartixBijection} is properly explained.

When $v_2\left(n\right)=v_2\left(t\right)-1$, $c_{2^v}\left(n\right)=-2^{v-1}$. The equation above becomes $Q_{m,t}\left(n\right)=\left|\mathcal{N}_{m,t,\left<n\right>}\right|+\left|\mathcal{N}_{m,t,2^v\left<N\right>}\right|$, where the notation $2^v\left<N\right>$ means each member of $\left<N\right>$ is multiplied by $2^v$. Since $n=2^{v-1}N$ in this case, the sum above can be written as $\left|\mathcal{N}_{m,t,\left<2n\right>}\right|$. This shows that the third statement in Theorem \ref{Thm:NecklacePartixBijection} is true.

Finally, when $v_2\left(n\right)\geq v_2\left(t\right)$, then $c_{2^v}\left(n\right)=2^{v-1}$, and one has $Q_{m,t}\left(n\right)=\left|\mathcal{N}_{m,t,\left<n\right>}\right|-\left|\mathcal{N}_{m,t,2^v\left<N\right>}\right|$. Note that $\left|\mathcal{N}_{m,t,\left<n\right>}\right|$ really can be written as $\left|\mathcal{N}_{m,t,\left<2^vN\right>}\right|$ as the allowed frequencies are limited by the divisors of $t$ and so cannot be divisible by $2^{v+1}$. Then $\left|\mathcal{N}_{m,t,\left<n\right>}\right|-\left|\mathcal{N}_{m,t,2^v\left<N\right>}\right|=\left|\mathcal{N}_{m,t,\left<2^vN\right>}\right|-\left|\mathcal{N}_{m,t,2^v\left<N\right>}\right|=\left|\mathcal{N}_{m,t,\left<2^{v-1}N\right>}\right|$, which can be written as $\left|\mathcal{N}_{m,t,\left<\left(t,n\right)/2\right>}\right|$.

\subsection{Revisiting the Partition of the Identity Element}\label{SubSec:NecklaceIdentityElem}

It is a corollary of the first statement of Theorem \ref{Thm:NecklacePartixBijection} that the set of necklaces of $m$ beads is equal to the number of partitions of the identity element of the group $\mathbb{Z}/m\mathbb{Z}$ with distinct parts when $m$ is odd. This confirms the findings from previous authors. For the case when $m$ is even, it is now clear that there is still a bijection between the partitions counted by $\sum_{t}Q_{m,t}\left(0\right)$ and some subset of the set of necklaces of $m$ beads. It is just that one has to avoid counting the necklaces with frequency $u$, where $v_2\left(u\right)=v_2\left(t\right)$ whenever $v_2\left(m\right)\geq v_2\left(t\right)\geq 1$.

As an illustration, consider the case $m=8$. The number of possible necklaces is 36, and the number of partitions of the identity element of $\mathbb{Z}/m\mathbb{Z}$ with distinct parts is 32. To account for the difference of 4, one can consult Theorem \ref{Thm:NecklacePartixBijection} and notice that the necklaces with frequency equal to 2 should not be counted when $t=2$ and $t=6$, necklaces with frequency equal to 4 should not be counted with $t=4$, and necklaces with frequency equal to 8 should not be counted when $t=8$. There is one necklace meeting each of these 4 descriptions, as desired.

\begin{figure}[ht]
	\centering
	\includegraphics[scale=0.21]{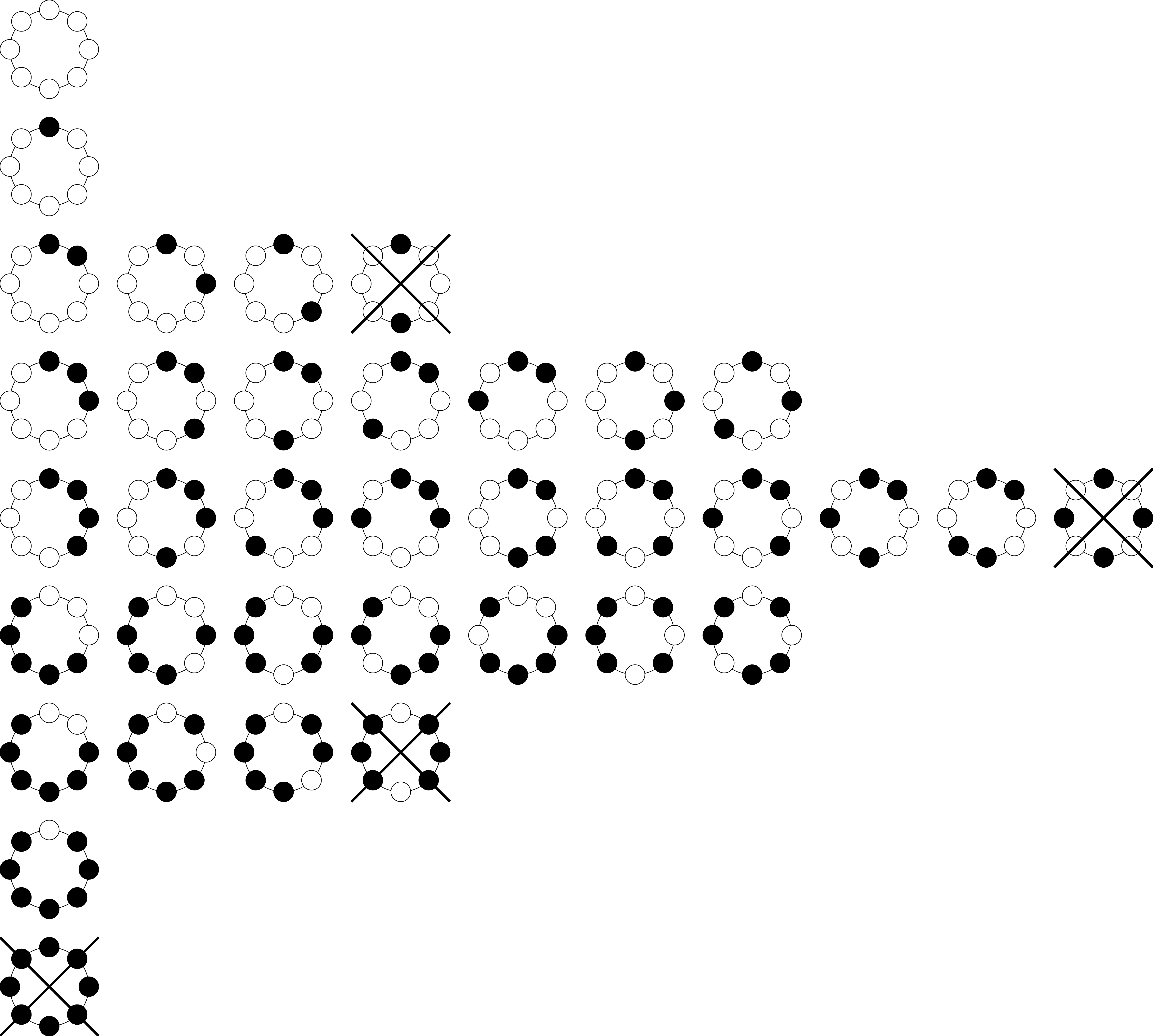}
	\caption{Necklaces with 8 beads.\label{fig:Necklace8}}
\end{figure}

\end{document}